\documentclass{article}

\usepackage[utf8]{inputenc}
\usepackage{amsmath}
\usepackage{amssymb}
\usepackage{amsfonts}
\usepackage{amsthm}
\usepackage{appendix}
\usepackage{hyperref}
\newcommand{\bun}{\mathrm{Bun}}
\newcommand{\B}{\mathrm{B}}
\newcommand{\D}{\mathrm{D}^b}
\newcommand{\map}{\mathrm{Map}}

\newcommand{\sll}{\mathrm{SL}_2}
\newcommand{\slr}{\mathrm{SL}_r}

\newcommand{\E}{\mathcal{E}}
\newcommand{\V}{\mathcal{V}}

\newcommand{\K}{\mathcal{K}}

\newcommand{\leng}{\ell}
\newcommand{\h}{h^\vee}

\newcommand{\gra}{\mathrm{gr}_{*}}
\newcommand{\sym}{\mathrm{Sym}}
\newcommand{\Hom}{\mathrm{Hom}}
\newcommand{\cdl}{\mathcal{L}}
\newcommand{\co}{\mathcal{O}}

\newcommand{\N}{\mathbb{N}}
\newcommand{\G}{\mathbb{G}}
\newcommand{\g}{\mathfrak{g}}
\newcommand{\ga}{G^{ad}}
\newtheorem{thm}{Theorem}
\newtheorem*{thm*}{Theorem}

\newtheorem{rmk}{Remark}
\newtheorem{lem}{Lemma}
\newtheorem{cor}{Corollary}
\newtheorem*{cor*}{Corollary}
\newtheorem{conj}{Conjecture}

\begin{document}

\title{Semiorthogonal decomposition of $\mathrm{D}^b(\mathrm{Bun}_2^L)$}

\author{Kai Xu\footnote{kaixu@math.harvard.edu}\,\, and Shing-Tung Yau\footnote{yau@math.harvard.edu}}
% \email{kaixu@math.harvard.edu}
% \affiliation{Mathematics Department, Harvard University}

% \begin{document}

\maketitle

\begin{abstract}
    We study the derived category of coherent sheaves on various versions of moduli space of vector bundles on curves by the Borel-Weil-Bott theory for loop groups and $\Theta$-stratification, and construct a semiorthogonal decomposition with blocks given by symmetric powers of the curve.
\end{abstract}

\section{Introduction}
Semiorthogonal decomposition of derived categories of coherent sheaves on algebraic varieties encodes rich information about the variety. It is shown that for a smooth projective variety if the canonical bundle is generated by global sections, there is no nontrivial semiorthogonal decomposition \cite{kawatani2015nonexistence}. In contrast, for Fano varieties there tends to be a large semiorthogonal decomposition, which reveals deep information of the geometric structures (for example, the Kutznesov component of cubic fourfold is closely related to the rationality problem \cite{kuznetsov2010derived}). Moduli theory is a natural source of interesting Fano varieties whose derived category is difficult to study directly. 

In this paper we will study a first example of such problem arising from the moduli of vector bundles on curves. Traditionally the coarse moduli (being an algebraic variety instead of a stack) is the central object of study, but we will see that in this example the study of fine moduli is very rewarding and brings new knowledge even to people who merely cares about the coarse moduli, which we believe is a universal phenomena for all moduli spaces. Some discussions around this phenomena in the context of Deligne-Mumford stacks also appeared elsewhere, for example \cite{polishchuk2019semiorthogonal}\cite{castravet2020derived1}\cite{castravet2020derived}.

In our situation, the fine moduli has a paricular nice structure of double quotient of loop group given by Weil uniformization (in this context proved by Drinfeld and Simpson \cite{drinfeld1995b}), which allows to compute the cohomology explicitly by the Borel-Weil-Bott theory of loop groups developed by C. Teleman in \cite{teleman1995lie}\cite{teleman1996verlinde}\cite{teleman1998borel}. This would lead to a semiorthogonal decomposition of $\mathrm{D}^b(\mathrm{Bun}_{\sll})=\mathrm{D}^b\mathrm{Coh}(\mathrm{Bun}_{\sll})$, the bounded derived category of coherent sheaves on the moduli stack of $\sll$ bundles over a fixed curve $X$. In fact, we will consider a slightly generalized version of $\sll$ bundles twisted by a $Z=Z(\sll)$ gerbe $\xi$, whose moduli will be denoted by $\mathrm{Bun}_{\sll}^\xi$ and is identified with the moduli $\mathrm{Bun}_2^L$ of rank two vector bundles with a fixed determinant $L$.

Our main theorem goes as follows (here $\mathcal{L}$ is a generator of $\mathrm{Pic}(\mathrm{Bun}_2^L)$): 

\begin{thm*}
We have the following semiorthogonal decomposition:
\begin{equation}
    \D(\mathrm{Bun}_2^L)=\langle \mathcal{L} ^{\otimes k}\otimes \D(\sym^n X), \mathcal{A} \rangle
\end{equation}
where $0\leq k<4$, $n\in \mathbb{N}$ and $\mathcal{A}$ is the right orthogonal complement of previous blocks (which is conjecturally zero).
\end{thm*}

To prove the main theorem it is natural to construct the blocks $\D(\sym^n X)$ from the universal $\sll$ bundle on $\bun_{\sll}^\xi$: for all $n\in\mathbb{N}$ let $X^{(n)}=X^{n}/S_{n}$ be the $n$-th stacky symmetric power of $X$, we have a natural twisted vector bundle $\E_n$ on $X^{(n)}\times \bun_{\sll}^\xi$, whose fibers are given by tensor products of fundamental representation $V$ of $\sll$ twisted by the $\sll$ bundle (see Section 2 for more precise definition). Such a twisted bundle gives rise to to a functor\footnote{A priori the domain should be the derived category of twisted sheaves, but since the Brauer group of curve is trivial this is equivalent to the usual derived category (non-canonically). We should work with twisted sheaf if we want a completely canonical functor. In particular, we need to descent the trivial gerbe on $X^{(n)}$ determined by $\xi$ \textit{canonically} to $\sym^nX$, which can be done because the action of isotropy groups in $X^{(n)}$ can indeed be canonically trivialized.} $\D(X^{(n)})\longrightarrow \D(\bun_{\sll}^\xi)$ which reveals rich structures of the target and are already very interesting when $n=1$. They give rise to embeddings $\D(\sym^n X)\longrightarrow \D(\bun_{\sll}^{\xi})$. To prove these properties the key technical input is the Borel-Weil-Bott theory for loop groups by C. Teleman which identifies the cohomology of various vector bundles on $\bun_G$ with conformal blocks, and allows us to compute them explicitly.

After we obtain the semiorthogonal decomposition for the fine moduli, the powerful machinery of $\Theta$-stratification developed by Halpern-Leistner  \cite{halpern2014structure}\cite{2020arXiv201001127H} (which is a vast generalization of the quantization commutes with reduction philosophy, cf. \cite{teleman2000quantization}\cite{halpern2015derived}\cite{ballard2012variation}) allows us (under the hypothesis that stability is equivalent to semistability) to extract a semiorthogonal decomposition for its stable part $\bun_G^{\xi,s}$ which is a gerbe over the coarse moduli $Bun_G^{\xi,s}$. By considering the character under the action of center we can obtain a decomposition of the coarse moduli (for the most interesting case for nontrivial $\xi$ in which $Bun_G^{\xi,s}$ is a smooth Fano variety of index  $2$). In more traditional langugage what we get is as follows:

\begin{cor*}
For line bundle $L$ with odd degree we have the following semiorthogonal decomposition:
\begin{equation}
    \D(Bun_2^{L,s})=\langle \mathcal{L}^{\otimes 2l}\otimes \D(\sym^n X), \mathcal{A}'' \rangle
\end{equation}
where $0\leq l<2$, $n<g-l$ and $\mathcal{A}''$ is the right orthogonal complement of previous blocks (which is conjecturally zero). Note $\theta=\mathcal{L}^{\otimes 2}$ is a generator of the Picard group of $Bun_2^{L,s}$.
\end{cor*}

This corollary has a long history with contributions from various groups of people\footnote{this may not be a complete list, we apologize for possible omission and welcome comments.}: the first one is the Bondal and Orlov result \cite{bondal1995semiorthogonal} in the genus $2$ case; Fonarev and Kuznetsov \cite{fonarev2018derived} proved the first nontrivial block\footnote{which is the third one, since the first two blocks follows from the fact that $Bun_2^{L,s}$ is a Fano manifold of index $2$} $\D(X)\longrightarrow \D(Bun_2^L)$ for generic curves by deformation from hyperelliptic curves; in \cite{narasimhan2016derived}\cite{narasimhan2017derived} Narasimhan proved the embedding $\D(X)\longrightarrow \D(Bun_2^L)$ for all $g\geq 2$ as well as the semiorthogonality of the first three blocks. Belmans, Galkin and Mukhopadhyay explicitly formulated this conjecture\footnote{with another part $\mathcal{A}''=0$ which is still open} in \cite{belmans2018semiorthogonal} and relates it with Grothendieck ring and mirror symmetry in \cite{belmans2020graph}; Lee \cite{lee2018remarks}, Gómez and Lee \cite{gomez2020motivic} stated an analogous conjecture for $Bun_n^L$ based on the motivic decompositioin (most explicitly for $n=3$). Belmans and Mukhopadhyay \cite{belmans2019admissible} proved a further generalization to arbitrary rank bundles the decomposition of first four blocks of this decomposition (for high genus curves). Lee and Narasimhan proved the next term $\D(\sym^2 X)\longrightarrow \D(Bun_2^L)$ (under certain mild assumptions) in their preprint \cite{lee2021symmetric}, and also introduced the Fourier-Mukai kernel for general $\D(\sym^n X)\longrightarrow \D(Bun_2^L)$ which is identical with what we use. During the preparation of our paper, this conjecture was proved in full generality independently by Tevelev and Torres using stable pairs \cite{tevelev2021narasimhan}. We will refer to this conjecture as BGMN (Belmans-Galkin-Mukhopadhyay and Narasimhan) conjecture to honor their significant contributions.

\section*{Acknowledgments} 

The authors are grateful for Swarnava Mukhopadhyay for introducing to us their works \cite{belmans2018semiorthogonal}\cite{belmans2020graph} and related history. K.X. is grateful to Yuchen Fu for great helps on geometric representation theory, and also thanks Shaoyun Bai, Lin Chen, Zhengping Gui and Ziquan Yang for helpful discussions.

\section{The main theorem}

Throughout this paper we fix an algebraically closed field $k$ of characteristic $0$ and a smooth projective curve $X$ over $k$. Let $G$ be a simple algebraic group\footnote{for this section we will consider $G=\sll$}, and $\bun_G=\map(X,\mathrm{B}G)$ be the moduli stack of $G$-bundles\footnote{which by definition means the stack representing the functor $S\mapsto \map(X\times S,\mathrm{B}G)$} over $X$. Let $Z=Z(G)$ denote the center of $G$, $\ga=G/Z$ denote the adjoint group of $G$, given a $Z$-gerbe $\xi\in \map (X,\B^2Z)$, we define the moduli of $\xi$-twisted $G$ bundle $\bun^\xi_G$ to be the fiber of $\map(X,\B \ga)\longrightarrow\map(X,\B^2Z)$ at $\xi$, which recover $\bun_G$ when $\xi$ is the trivial gerbe. Let $\bun_G^{\xi,s}$ (resp, $\bun_G^{\xi,ss}$) denote the open substack of stable (resp, semistable) twisted bundles. For $H=\slr$, we have an equivalent description $\bun^\xi_H\simeq \bun_r^L$ as the moduli stack of rank $r$ bundles with a fixed determinant $L$, where the degree of $L$ equals to $\xi$ in the cohomology group $H^2(X,Z)\simeq \mathbb{Z}/r\mathbb{Z}$.\footnote{Here we use the étale topology on $X$. Note that this a canonical isomorphism without choosing a root of unity.} When $(r,\deg L)=1$, it's well known that the coarse moduli of stable bundles $Bun_r^{L,s}$ is a smooth projective variety and the moduli stack of stable bundles $\bun_r^{L,s}$ is a $Z$ gerbe over the coarse moduli $Bun_r^{L,s}$.

Let $X^{(n)}=  X^{n}/S_{n}$ be the $n$-th stacky symmetric power of $X$ and $\sym^nX$ be its coarse moduli space (which is smooth and projective of dimension $n$). A fact we will use is that the natural pull back $\D(\sym^nX)\longrightarrow \D(X^{(n)})$ is fully-faithful (essentially because taking $S_n$-invariance is exact, a much stonger statement was proved in \cite{polishchuk2019semiorthogonal}), giving a full subcategory which will play an important role in this paper.

For any $n$, we have a representation $V^{\boxtimes n}$ of $G^{n}$ which is the external tensor product of $n$ copies of fundamental representation $V$. This determines a vector bundle on $  \B G^{n}$ (and a twisted sheaf on $  \B (\ga)^{n}$). We can pull it back by the evaluation map $X^{n}\times \bun^\xi_G\longrightarrow \B (\ga)^{n}$ and get a twisted sheaf $\Tilde{\V} _n$ on $X^{n}\times \bun^\xi_G$. This is naturally equivariant with respect to $S_n$ and hence descends to a twisted sheaf $\V_n$ on $X^{(n)}\times \bun^\xi_G$, which is equivalent to a functor $V_n: \D(X^{(n)},\chi\circ\xi)\longrightarrow \D(\bun^\xi_G)$ where $\chi:Z^{n}\longrightarrow \G_m$ is the product of central characters of $V$. As mentioned before, the domain may be identified with non-twisted derived category for the triviality of Brauer group for curves. The composition of $V_n$ with the natural embedding $\D(\sym^{n} X)\longrightarrow \D(X^{(n)})$ will be denoted by $E_n$. As we will show, $E_m$ is fully faithful for all $m$ and the image of $E_m$ is left orthogonal to the image of $E_n$ for $m<n$.

To prove the fully-faithfulness we need to invoke the celebrated Bondal-Orlov criterion for fully-faithfulness which reduces the proof to computation of the cohomology of certain vector bundles. The original version was proved in \cite{bondal1995semiorthogonal}, which already satisfy our need, if we want to work canonically with twisted sheaves we may use the generalization proved in \cite{lim2020bondal}.

\begin{thm}
Let $Y$ be a smooth projective variety over an algebraically closed field of characteristic $0$, $F$ is a functor from $\D(Y)$ to another triangulated category $\mathcal{T}$, then $F$ is fully faithful if and only if it admit a right adjoint, a left adjoint $G$ with $G\circ F$ of Fourier-Mukai type, and that \begin{itemize}
    \item for any closed point $y\in Y$ we have $$\Hom_\mathcal{T}(F(\co_y),F(\co_y))\simeq k$$
    \item for any closed points $y,z\in Y$ we have 
    $$\Hom_\mathcal{T}(F(\co_y),F(\co_z)[n])\simeq 0$$ unless $y=z$ and $0\leq n\leq \dim Y$
\end{itemize}
\end{thm}

We can actually compute the image of $\co_x$ under the functors $E_n$ up to extension. It turns out they have a very simple description as follows:
\begin{thm}
Given any point $x=(x_1,\cdots,x_{n})\in \sym^nX$, choose a preimage $x'\in X^{n}$, let $V_n$ be the functor\footnote{we used the same notation before for the same functor with natural $S_n$ equivariance condition} corresponding to the Fourier-Mukai kernel $\V_n\in\D(X^{n}\times \bun_G^\xi)$ we have a canonical increasing filtration $F^*$ on $E_n(\co_x)$ such that $$\gra\,\,E_n(\co_x)\simeq V_n(\co_{x'})$$ Moreover, the highest weight component of right hand side appears in the highest degree on the left hand side. 
\end{thm}

\begin{proof}
We have natural projections $\pi_1:X^{n}\times \bun_G^\xi\longrightarrow X^{(n)}\times \bun_G^\xi$ and $\pi_2: X^{(n)}\times \bun_G^\xi\longrightarrow \sym^nX\times \bun_G^\xi$. This theorem amounts to give a natural filtration on $\pi_{2*}(\pi_2^{\,\,*} (\co_{x})\otimes \V_n)$, which comes from the $\mathfrak{m}_{\pi_2^{-1}(x)}$-adic filtration on $\pi_2^{\,\,*} (\co_{x})$, and show $$\pi_{2*}(\pi_{1*}(\co_{x'})\otimes \V_n)\simeq \gra \pi_{2*}(\pi_2^{\,\,*} (\co_{x})\otimes \V_n)$$

In terms of representations of isotropy group $S_n$, $\pi_{1*}$ amounts to take coinduction, i.e. tensor with the regular representation, and $\pi_{2*}$ amounts to take invariant. Hence we only need to show that $\pi_2^{\,\,*} (\co_{x})$ is isomorphic to the regular representation of the isotropy group. This is a classical theorem of invariant theory \cite{chevalley1955invariants} (where $\pi_2^{\,\,*} (\co_{x})$ is called the coinvariant algebra).  The last statement follows from examination of the filtration structure.
\end{proof}

\begin{cor}
We have a natural exact sequence \begin{equation}
    0\longrightarrow E_n'(\co_x)\longrightarrow E_n(\co_x)\longrightarrow E_n''(\co_x)\longrightarrow 0
    \end{equation}where $E_n''(\co_x)$ is isomorphic to the highest weight irreducible component of $V_n(\co_x)$ and $E_n'(\co_x)$ has graded pieces isomorphic to lower weight components of $V_n(\co_x)$ and is generated by essential image of $V_m$ (hence $E_m$) for $m<n$.
\end{cor}

\begin{rmk}
We can really define $E_n'$ and $E_n''$ as Fourier-Mukai transform from certain strata of $X^{(n)}$ to $\bun_G^\xi$ so that $E_n'(\co_x)$ and $E_n''(\co_x)$ is what we define above. We will not use them hence will not spell out the details.
\end{rmk}

We will compute the homomorphisms by the Borel-Weil-Bott theorem for loop groups. To state this theorem (which holds for general semi-simple groups, though here we only use it for $G=\sll$), let us first set up our notations. Fix an integral level $c$, we call an irreducible representation $V$ of $G$ either regular or singular according to the location of the weight $(\lambda+\rho,\h+c)$ where $\lambda$ is the highest weight of $V$, $\rho$ is half-sum of the positive roots, $\h$ is the dual Coxeter number of $G$ (which is $2$ for $\sll$) and $c$ is the level: we say it's regular if it lies in the interior of some Weyl alcove and define $\leng(V)$ to be the smallest length of affine Weyl transform relating this Weyl alcove with the positive one and denote $V^{sm}$ to be the irreducible representation whose highest weight is the image of $\lambda$ in the positive Weyl alcove; otherwise when $(\lambda+\rho,\h+c)$ lies on the boundary of a affine Weyl alcove we say this representation is singular.

On $\bun^\xi_G$ there are a number of natural vector bundles to construct: given $N$ representations $V_i$ of $G$, we can pull it back along the natural evaluation map $X^{N}\times \bun^\xi_G\longrightarrow \B (G^{ad})^N$ and get a natural twisted sheaf on $X^{N}\times \bun^\xi_G$, which is only twisted in $X^N$ direction. Fix a point $(x_i)\in X^N$ we may consider its fiber which gives a vector bundle $\V$ on $\bun^\xi_G$. There is also a basic line bundle $\cdl$ which generates its Picard group.

Following is one form of the Borel-Weil-Bott theorem for loop groups proved\footnote{He worked in the untwisted case but parallel proof works with twisting, as we will sketch in the appendix; also he proved over complex numbers but this theory can be extended to algebraic setting by Lefschetz principle.} by C. Teleman in \cite{teleman1998borel}, which allows us to compute the cohomology of various tautological bundles on $\bun^\xi_G$:
\begin{thm}
$H^*(\bun^\xi_G,\cdl^{\otimes c}\otimes \V)$ vanishes if some $V_i$ is singular and is concentrated in degree $\sum\ell(V_i)$ if all of them are regular, in this case it's isomorphic to the global section $\Gamma(\bun^\xi_G,\cdl^{\otimes c}\otimes \V^{sm})$ where $\V^{sm}$ is the vector bundle obtained from representations $V_i^{sm}$.
\end{thm}

\begin{rmk}
Along the same lines we can actually show that for $-2\h<c<0$ all cohomology groups vanish. This will also be used for our semiorthogonal decomposition.
\end{rmk}

Now let's prove our key results. First we will prove the semiorthogonality.

\begin{lem}
For all $m>n$, let $A\in\D(\sym^m X)$, $B\in\D(\sym^n X)$, we have
$$\Hom(E_m(A),E_n(B))=0$$
\end{lem}

\begin{proof}
It suffices to check this on skyscraper sheaves $A=\co_{x},B=\co_y$ (as they are generators of derived category.) By Borel-Weil-Bott theoem we only need to consider the case when $y\equiv x \mod\, 2$ as divisors. Let us first prove that $\Hom(E_m(\co_x),V_n(\co_y))=0$ for all such $y< x$. Again by Borel-Weil-Bott theory, this is reduced to computing the homomorphism (of corresponding representations) in the category of $(\g[t],G)$-modules (here $G=\sll$), and we may also assume $x$ and $y$ are both supported at one point $0\in \mathbb{A}^1$. In this case $E_m$ corresponds to the representation given by the construction
$R_m=\sym^m (V[t])/(\sigma_1,\sigma_2,\cdots,\sigma_m)\sym^m V[t]$ where $\sigma_i$ are fundamental symmetric polynomials of $t_i$, and $V_n$ corresponds to the representation given by the construction $S_n=\sym^n (V[t])/(t_1,t_2,\cdots,t_n)\sym^n V[t]$. By same arguments with Theorem 2, $R_m$ is identified with $V^{\otimes m}$ as a representation of $\sll$. Moreover we can directly see that the highest weight component $\sym^{m}V$ generates the whole space as a $\g[t]$-module by choosing a nice basis. As $S_n$ receives no nonzero $\sll$-equivariant map from $\sym^{m}V$, it receives no nonzero $\g[t]$-equivariant map from $R_m$. To prove higher degree homomorphisms also vanish, we use induction on $m+n$, and may assume $\Hom^*(E_m(A),E_n(B))$ is supported on the diagonal of $\sym^m X\times \sym^n X$. Moreover, Borel-Weil-Bott theory shows that the cohomological amplitudes of its stalks concentrate in $[0,\frac{m+n}{2}]$. By the previous arguments, there are no homomorphisms of degree $0$, hence the cohomological amplitudes have length at most $\frac{m+n}{2}-1$, which is smaller than the codimension $m+n-1$ of the diagonal, hence $\Hom^*(E_m(A),E_n(B))=0$.

% The claim follows from the fact that if we consider the decomposition of $V_n$ (here we identify the Fourier-Mukai transformation with its kernel) under the semiorthogonal decomposition of $\D(X^{(n)}\times \bun^\xi,\chi\circ\xi)\simeq \D(X^{(n)},\chi\circ\xi)\otimes\D( \bun^\xi)$ induced from Theorem 2, the rightmost piece being $E_n$, and all other graded pieces are either $0$ or $E_m(\co_y)$ for $y\equiv x \mod 2$ and $y<x$. (And all such $y$ appear.) We may compute that $\Hom(E_n(\co_x),E_m(\co_y))=\Hom(i_x^*(\mathcal{E}_n),i_y^*(\mathcal{E}_m))=\Hom_{\D(X^{(n)}\times \bun^\xi,\chi\circ\xi)}(\mathcal{E}_n,i_{x*}(i_y^*(\mathcal{E}_n)\otimes \mathrm{Alt}))=0$ because $\mathcal{E}_n$ lies in the rightmost piece while $i_{x*}(i_y^*(\mathcal{E}_m)\otimes \mathrm{Alt})$ lies in another component (to the left).

Now it suffices to show that we can always factor a map $E_m(\co_x)\longrightarrow E_n(\co_y)$ through the subcategory generated by $V_l(\co_z)$ (or equivalently $E_l(\co_z)$) where $z<x$ as divisors. Again by Theorem 2 and the following lemma we may replace $E_m,E_n$ by $V_m,V_n$ and this follow from the Borel-Weil-Bott theorem.   
\end{proof}

\begin{lem}
For all triangulated category $\mathcal{C}$ and objects $E,F$ with finite filtration, if all homomorphisms from $\gra\, E$ to $\gra\, F$ factor through a triangulated subcategory $\mathcal{D}$, then all homomorphisms from $E$ to $ F$ also factor through $\mathcal{D}$.
\end{lem}

\begin{proof}
When $\mathcal{D}=\{0\}$ this is follows from a spectral sequence argument, in general just pass to the Verdier quotient $\mathcal{C}/\mathcal{D}$.
\end{proof}

Then we proceed to fully-faithfulness: 

\begin{thm}
The functors $E_n$ are fully-faithful.
\end{thm}

\begin{proof}

For $n=0$, we have $E_0(\co _\varnothing)\simeq \co_{\bun^\xi_G}$, thus fully-faithfulness amounts to computing some of the Hodge numbers of $\bun^\xi_G$, which are well-known: $h^{0,0}=1$ and $h^{i,0}=0$ for $i>0$.

 We now assume that we have proved the theorem for $m<n$. then we can see that in the long exact sequence (3) $E_n'$ lies in the full subcategory generated by $E_m$ for $|m|<|n|$ and hence $\Hom(E_n(\co_x),E_n(\co_x))=\Hom(E_n(\co_x),E_n''(\co_x))$ by the semiorthogonality. This allows us to apply Borel-Weil-Bott theorem because $E_n''(\co_x)$ comes from an irreducible representation. We have the following exact triangle 
 \begin{align*}
    & \Hom^*(E_n''(\co_x),E_n''(\co_x))\longrightarrow\Hom^*(E_n(\co_x),E_n''(\co_x))\\
     &\longrightarrow\Hom^*(E_n'(\co_x),E_n''(\co_x))\longrightarrow \Hom^*(E_n''(\co_x),E_n''(\co_x))[1]
 \end{align*}

By Borel-Weil-Bott theorem we can estimate their degree: $\Hom^*(E_n'(\co_x),E_n''(\co_x))$ lies in degree $[1,n]$, $\Hom^*(E_n''(\co_x),E_n''(\co_x))$ lies in degree $[0,n]$ and the cohomology is $k$ in degree $0$. Therefore $\Hom^*(E_n(\co_x),E_n''(\co_x))$ lies in degree $[0,n]$ and the coholomogy is $k$ on in degree $0$. By Orlov's criterion we see that $E_n$ is fully faithful.
\end{proof}

Hence we get the semiorthogonal decomposition of $\D(\bun_G^\xi)$:

\begin{thm}
For $G=\sll$ we have the following semiorthogonal decomposition
\begin{equation}
    \D(\bun_G^\xi)=\langle \mathcal{L} ^{\otimes k}\otimes \D(\sym^n X), \mathcal{A} \rangle
\end{equation}
where $0\leq k< 2h^\vee=4$, $n\in\N$ and $\mathcal{A}$ is the right orthogonal complement of previous blocks.
\end{thm}

Now we proceed to study the substack $\bun_G^{\xi,ss}$ of semistable bundles. By the machinery of $\Theta$-strafication developed by D. Halpern-Leistner (cf. \cite{halpern2014structure}\cite{2020arXiv201001127H}) the Harder-Narasimhan-Shatz stratification on $\bun_G^\xi$ gave rise to a natural baric structure on $\D(\bun_G^\xi)$. It's plausible to conjecture that our semiorthogonal decomposition can be related to this one by mutation. In particular, this would implies the following theorem that the beginning terms in our semi-orthogonal decomposition naturally embeds into $\D(\bun_G^{\xi,ss})$, which may be checked directly. (We only consider the most interesting case with nontrivial $\xi$ for simplicity, where $\bun_G^{\xi,ss}\simeq \bun_G^{\xi,s}$)

\begin{thm}
For $G=\sll$ and nontrivial $\xi$ we have the following semiorthogonal decomposition:
\begin{equation}
    \D(\bun_G^{\xi,ss})=\langle \mathcal{L} ^{\otimes k}\otimes \D(\sym^n X), \mathcal{A}' \rangle
\end{equation}
where $0\leq k< 2h^\vee=4$, $n<g-[\frac{k}{2}]$ and $\mathcal{A}'$ is the right orthogonal complement of previous blocks.
\end{thm}
\begin{proof}
The powerful machinery of \cite{2020arXiv201001127H} reduces this statement to a routine computation, by Theorem 2.2.2 we only need to bound the weight of natural $\G_m$ action on $E_n$ restricted to unstable locus (or more precisely differents between weights) by the weight of normal bundle (which was denoted by $\eta$ in \textit{loc. cit.}). This can be done directly and we can see that $n<g-[\frac{k}{2}]$ gives a maximal collection of subcategories from the decomposition which can be embedded into the stable locus.
\end{proof}

We may rigidify the $\B Z$ action on $\bun_G^{\xi,ss}$ and get the following decomposition for the coarse moduli $Bun_G^{\xi,ss}$:

\begin{thm}
For $G=\sll$ and nontrivial $\xi$ we have the following semiorthogonal decomposition:
\begin{equation}
    \D(Bun_G^{\xi,ss})=\D(\bun_G^{\xi,ss})^{\B Z}=\langle \mathcal{L}^{\otimes 2l}\otimes \D(\sym^n X), \mathcal{A}'' \rangle
\end{equation}
where $0\leq l<2$, $n<g-l$ and $\mathcal{A}''$ is the right orthogonal complement of previous blocks.
\end{thm}

\begin{proof}
This follows from the general fact that $\D(X/\mathcal{G})\simeq\D(X)^\mathcal{G}$, i.e. the derived category of a quotient (with respect to the abelian groupstack $\B Z$ here) is equivalent to the invariant of the original derived category, and being invariant under $\B Z$ is equivalent to be fixed by $Z$ (which is an automorphism for all points in $\bun_G^{\xi,s}$)
\end{proof}

This theorem partially proves the BGMN (Belmans-Galkin-Mukhopadhyay and Narasimhan) conjecture.

\section{Generalizations to simple groups}

Our construction in the previous section only works for $G=\sll$ because we used essentially the property that the irreducible representations are precisely symmetric powers of the fundamental representation. For more general simple and simply connected\footnote{this not really a restriction, twisted bundles for $G$ can be always be viewed as twisted bundles for its universal cover, stability properties remaining unchanged} $G$ the symmetric powers are no longer irreducible but the fundamental representations still gives rise to embeddings of $\D(X)$ into $\D(\bun^\xi_G)$. There are $2h^\vee\times r$ copies of $\D(X)$ (where $r$ is the rank\footnote{note that the rank of vector bundle and cooresponding simple group differ by $1$, which may cause confusion.} of $G$ and $h^\vee$ is the dual Coxter number), by combining $r$ different fundamental representations with different powers of $\cdl$, and also $2h^\vee$ copies of $\D(pt)$ coming from different powers of $\cdl$. When the genus of $X$ is greater than $1$ and $\bun_G^{\xi,ss}\simeq \bun_G^{\xi,s}$ (for example, when $G$ is of type A and $\xi$ is a generator of $H^2(X,Z)$ \footnote{we don't know if this is still true for general $G$}), we can further obtain that all these blocks can be embedded into $\D(\bun_G^{\xi,s})$, among which (at least, exact for type A) two copies of $\D(pt)$ and $2r$ copies of $\D(X)$ are invariant under $\B Z$. Hence we have the following theorem (which generalizes the result in \cite{belmans2019admissible}):

\begin{thm} For simple group $G$ with rank $r$ we have a semiorthogonal decomposition:
\begin{equation}
    \D(\bun_G^\xi)=\langle \cdl ^{\otimes k}\otimes \D(pt), \cdl ^{\otimes k}\otimes\D(X)_i,\mathcal{A} \rangle
\end{equation}
where $0\leq k< 2h^\vee$, $i=1,\cdots,r$ label the fundamental representations of $G$ and $\mathcal{A}$ is the right orthogonal complement of previous blocks.

Furthermore, when the genus of $X$ is greater than $1$ and $G$ is of type A, for twisting parameter $\xi\in H^2(X,Z)$ generating $H^2(X,Z)$ (and thus $\bun_G^{\xi,s}$ is a smooth Fano variety), we have the following 
\begin{equation}
    \D(Bun_G^{\xi,s})=\D(\bun_G^{\xi,s})^{\B Z}=\langle \cdl^{\otimes h^\vee l}\otimes \D(pt), \cdl^{\otimes h^\vee l}\otimes\D(X)_i, \mathcal{A}''\rangle
\end{equation}
where $0\leq l<2$, $i=1,\cdots,r$ and $\mathcal{A}''$ is the right orthogonal complement of previous blocks. (Note here $\theta=\cdl^{\otimes h^\vee }$ is a generator of $\mathrm{Pic}(Bun_G^{\xi,s})$)
\end{thm}

Based on the motivic decomposition of $\bun_G$ \cite{behrend2007motivic}, previous results \cite{lee2018remarks} \cite{gomez2020motivic} and our main theorem for $G=\sll$, it's natural to make the following conjecture: 

\begin{conj}
For simple group $G$ with rank $r$ we have a semiorthogonal decomposition:
\begin{equation}
    \D(\bun_G^\xi)=\langle \cdl ^{\otimes k}\otimes \D(\sym^\mathbf{n} X)\rangle
\end{equation}
where $0\leq k< 2h^\vee$, $\mathbf{n}=(n_1,\cdots,n_r)\in \mathbb{N}^r$ and $\sym^\mathbf{n} X=\sym^{n_1}X\times\cdots\times\sym^{n_r}X$. Furthermore, for type A group with twisting parameter $\xi\in H^2(X,Z)$ if all semistable bundles are stable, we have a semiorthogonal decomposition:
\begin{equation}
    \D(Bun_G^{\xi,s})=\langle \cdl^{\otimes h^\vee l}\otimes \D(\sym^\mathbf{n} X)\rangle
\end{equation}
where $0\leq l<2$, $\mathbf{n}\in \mathbb{N}^r$ such that $l,\mathbf{n}$ and the genus $g$ satisfy a linear inequality (with constant term) depending on $G$. 
\end{conj}

\begin{rmk}
The difference between the prefactors in the motivic decomposition and in the semiorthogonal decomposition is explained by the fact that $\bun_G^\xi$ is not a variety but a stack. The derived category is more closely related to the $K$ theory, which is isomorphic to the cohomology (when tensoring with $\mathbb{Q}$) for varieties but not for stacks. For stacks the decomposition of derived category should be related to the motivic decomposition of its inertia stack.
\end{rmk}

However, if we naively mimic the construction in the previous section by taking the product of symmetric powers of fundamental representations, the Fourier-Mukai transforms we get are not embeddings because products of symmetric powers of fundamental representations are not irreducible representation in general. We still need to take the semiorthogonal projection onto the complement of smaller blocks, by induction, which give rise to a complicated functor serving as a candidate for the correct embedding. We leave further discussion of this question for future works.

\appendix
\section{Borel-Weil-Bott theory for twisted bundles}

In this paper we used a variant of Borel-Weil-Bott theory developed by C. Teleman in \cite{teleman1995lie}\cite{teleman1996verlinde}\cite{teleman1998borel} for twisted bundles, which was not covered in the previous works but follows from parallel argument. In this appendix we will give a sketch how the proof goes, focusing on the modifications we make for the twisting.

The computation of coherent cohomology in \cite{teleman1998borel} on $\bun_G$ is based on the celebrated Weil uniformization $\bun_G\simeq G(\co_{X-x})\backslash G(\K_x)/G(\co_x)$ (we will only write the single punctured version for simplicity). For twisted bundles we have the following analogue: by standard theory of loop groups (cf. \cite{pressley1988loop}) $\xi\in H^2(X,Z)\simeq Z$ determines an outer automorphism of the loop group $G(\K_x)$ and let $G(\co_x)^{\xi}$ be the image of $G(\co_x)$ under this outer-automorphism (which is well defined up to conjugation), then we may write the moduli of twisted bundles as a double quotient $\bun^\xi_G\simeq G(\co_{X-x})\backslash G(\K_x)/G(\co_x)^{\xi}$. Here $G(\K_x)/G(\co_x)^{\xi}$ as a homogenous space of loop group may be viewed as the usual affine Grassmannian $G(\K_x)/G(\co_x)$ with a twisted action by the outer automorphism $\xi$. 

Let us briefly recall the calculation in \cite{teleman1998borel}: to compute the cohomology of vector bundles on $\bun_G\simeq G(\co_{X-x})\backslash G(\K_x)/G(\co_x)$ we may first compute the cohomology of homogenous vector bundles on the affine Grassmannian $G(\K_x)/G(\co_x)$, which has nothing to do with the curve $X$ and is well known to representation theorists (cf. \cite{kumar1987demazure}), this gives us a loop group representation and then we compute the Lie group cohomology with respect to $G(\co_{X-x})\subset G(\K_x)$, which was done in \cite{teleman1998borel}.

In our twisted setting, the only difference is that we need to twist the representation by $\xi$ before taking cohomology with respect $G(\co_{X-x})$, which is easily described by loop group representation theory. For our purpose this is even simpler: in this paper we only need to consider the case when $-2h^\vee<c\leq 0$, when $-2h^\vee<c<0$ the representation we get is always $0$, and when $c=0$ the only irreducible representation is the trivial one. Hence the twisting does not make any difference, which explains why our main theorem does not depend on the twisting $\xi$ we use. (But different $\xi$ gives rise to drastically different $\Theta$-stratification, hence the derived category of stable bundles depends essentially on  the twisting $\xi$).

\bibliographystyle{plain}
\bibliography{ref}
\end{document}